\documentclass[12pt,english]{amsart}
\usepackage{amsmath,amsthm,amssymb,amsfonts,amscd, array,latexsym}
\usepackage[noadjust]{cite}
\usepackage[T2A]{fontenc}
\usepackage[cp1251]{inputenc}
\usepackage[english]{babel}
\usepackage{diagmac2} %% for diagrams
\usepackage{tikz}
\usetikzlibrary{positioning,calc,arrows,arrows.meta}

\usepackage{color} %for using colors

\newtheorem{theorem}{Theorem}[section]
\newtheorem{corollary}[theorem]{Corollary}
\newtheorem{lemma}[theorem]{Lemma}
\newtheorem{proposition}[theorem]{Proposition}

\theoremstyle{definition}
\newtheorem{definition}[theorem]{Definition}
\newtheorem{example}[theorem]{Example}

\theoremstyle{remark}
\newtheorem{remark}[theorem]{Remark}
\numberwithin{equation}{section}

\makeatletter

\renewcommand{\p@enumii}{}

\makeatother

\newcommand{\RR}{\mathbb{R}}
\newcommand{\NN}{\mathbb{N}}

\newcommand{\Fr}{\operatorname{Fr}}

\def\<#1>{\langle #1 \rangle}

\newbox\onebox
\newcommand{\coherent}[1]{\mathbin{\setbox\onebox=\hbox{$=$}\lower0.7\ht%
\onebox\hbox{$\stackrel{#1}{=}$}}}

\newcommand{\acr}{\newline\indent}

\newcommand{\CEC}{\mathbf{CEC}}

\newcommand{\ol}[1]{\overline{#1}}
\newcommand{\mbf}[1]{\mathbf{#1}}

\newcommand{\seq}[3][\mathbb{N}]{(#2)_{#3\in #1}}

\begin{document}

\title[Completeness, Closedness and Metric Reflections]{Completeness, Closedness and Metric Reflections of Pseudometric Spaces}

\author{Viktoriia Bilet}
\address{\textbf{Viktoriia Bilet}\acr
Department of Theory of Functions \acr
Institute of Applied Mathematics and Mechanics of NASU\acr
Dobrovolskogo str. 1, Slovyansk 84100, Ukraine}
\email{viktoriiabilet@gmail.com}

\author{Oleksiy Dovgoshey}
\address{\textbf{Oleksiy Dovgoshey}\acr
Department of Theory of Functions \acr
Institute of Applied Mathematics and Mechanics of NASU \acr
Dobrovolskogo str. 1, Slovyansk 84100, Ukraine \acr
Institut fuer Mathematik Universitaet zu Luebeck\acr
Ratzeburger Allee 160, D-23562 Luebeck, Deutschland}

\email{oleksiy.dovgoshey@gmail.com}



\begin{abstract}
It is well-known that a metric space \((X, d)\) is complete iff the set \(X\) is closed in every metric superspace of \((X, d)\). For a given pseudometric space \((Y, \rho)\), we describe the maximal class \(\CEC(Y, \rho)\) of superspaces of \((Y, \rho)\) such that \((Y, \rho)\) is complete if and only if \(Y\) is closed in every \((Z, \Delta) \in \CEC(Y, \rho)\).

We also introduce the concept of pseudoisometric spaces and prove that spaces are pseudoisometric iff their metric reflections are isometric. The last result implies that a pseudometric space is complete if and only if this space is pseudoisometric to a complete pseudometric space.
\end{abstract}

% -----------------------------------------------------------
\maketitle
% -----------------------------------------------------------

\section{Introduction and Preliminaries}

The present paper is aimed to expand some well-known characterizations of complete metric spaces to pseudometric ones. Let us start from the classical notion of metric space introduced by Maurice Fr\'{e}chet in his thesis \cite{Fre1906RdCMdP}.

In what follows, we will denote by \(\NN\) the set of all positive integer numbers, and \(\RR\) the set of all real numbers.

A \textit{metric} on a set \(X\) is a function \(d\colon X^{2} \to \RR\) such that for all \(x\), \(y\), \(z \in X\):
\begin{enumerate}
\item \(d(x,y) \geqslant 0\) with equality if and only if \(x=y\), the \emph{positivity property};
\item \(d(x,y)=d(y,x)\), the \emph{symmetry property};
\item \(d(x, y)\leq d(x, z) + d(z, y)\), the \emph{triangle inequality}.
\end{enumerate}

In 1934 \DJ{}uro Kurepa \cite{Kur1934CRASP} introduced the pseudometric spaces which, unlike metric spaces, allow the zero distance between different points.

\begin{definition}\label{ch2:d2}
Let \(X\) be a set and let \(d \colon X^{2} \to \RR\) be a non-negative, symmetric function such that \(d(x, x) = 0\) for every \(x \in X\). The function \(d\) is a \emph{pseudometric} on \(X\) if it satisfies the triangle inequality.
\end{definition}

If \(d\) is a pseudometric on \(X\), we say that \((X, d)\) is a \emph{pseudometric} \emph{space}.

Let us recall some concepts connected with closedness and completeness in pseudometric spaces. The definitions below are standard and can be found, for example, in \cite{Fre2014} and \cite{Kelley1975}.

Let \((X, d)\) be a pseudometric space. An \emph{open ball} with a \emph{radius} \(r > 0\) and a \emph{center} \(c \in X\) is the set
\[
B_r(c) = \{x \in X \colon d(c, x) < r\}.
\]

Now we can introduce a topology on pseudometric spaces using the open balls.

\begin{definition}\label{d2.4}
Let \((X, d)\) be a pseudometric space, \(\mbf{B}_X = \mbf{B}_{X, d}\) be the set of all open balls in \((X, d)\) and \(\tau\) be a topology on \(X\). We will say that \(\tau\) is generated by pseudometric \(d\) if \(\mbf{B}_X\) is an open base for \(\tau\).
\end{definition}

Thus, \(\tau\) is generated by \(d\) if and only if every \(B \in \mbf{B}_X\) belongs to \(\tau\) and every nonempty \(A \in \tau\) is the union of a family of elements of \(\mbf{B}_X\).

In what follows, we assume that every pseudometric space \((X, d)\) is also a topological space endowed with the topology \(\tau\) generated by \(d\).

As in the case of the metric spaces, we define the notion of a Cauchy sequence as follows.

\begin{definition}
Let \((X, d)\) be a pseudometric space. A sequence \((x_n)_{n \in \mathbb{N}} \subseteq X\) is a \emph{Cauchy sequence} in \((X, d)\) if, for every \(r > 0\), there is an integer \(n_0 \in \NN\) such that \(x_n \in B_r(x_{n_0})\) for every \(n \geqslant n_0\).
\end{definition}

\begin{remark}\label{r2.17}
Here and later the symbol \((x_n)_{n\in \mathbb{N}} \subseteq X\) means that \(x_n \in X\) holds for every \(n \in \mathbb{N}\).
\end{remark}

A sequence \((x_n)_{n \in \mathbb{N}}\) of points in a pseudometric space \((X, d)\) is said to \emph{converge to} a point \(a \in X\) if
\begin{equation*}
\lim_{n \to \infty} d(x_n, a) = 0.
\end{equation*}
A sequence is convergent if it is \emph{convergent} to some point. It is clear that every convergent sequence is a Cauchy sequence.

\begin{definition}\label{d2.6}
A subset \(S\) of a pseudometric space \((X, d)\) is \emph{complete} if every Cauchy sequence \((x_n)_{n \in \mathbb{N}} \subseteq S\) converges to a point \(s \in S\).
\end{definition}

The purpose of the second section of the present paper is to obtain ``pseudometric'' modifications of some results known for complete metric spaces. In particular, we find generalizations of the following theorems (see Theorems~10.2.1 and 10.3.1 in \cite{Sea2007}).

\begin{theorem}\label{t1.10}
Let \((Y, \rho)\) be a metric space. The set \(Y\) is closed in every metric superspace of \((Y, \rho)\) iff \((Y, \rho)\) is complete.
\end{theorem}

Here and what follows, for pseudometric spaces \((Y, \rho)\) and \((X, d)\), we say that \((X, d)\) is a superspace of \((Y, \rho)\) if \(Y \subseteq X\) and \(d|_{Y^2} = \rho\) hold.

\begin{theorem}\label{t1.11}
Let \((X, d)\) be a complete metric space and let \(S\) be a subset of \(X\). Then \(S\) is closed iff \(S\) is complete.
\end{theorem}

For every pseudometric space \((Y, \rho)\) we define the class \(\CEC(Y, \rho)\) (\textbf{C}ompleteness \textbf{E}quivalent \textbf{C}losedness) of pseudometric spaces as follows.

\begin{definition}\label{d2.12}
Let \((Y, \rho)\) be a pseudometric space. A pseudometric space \((X, d)\) belongs to \(\CEC(Y, \rho)\) if \((X, d)\) is a superspace of \((Y, \rho)\) and the inequality
\[
d(x, y) > 0
\]
holds whenever \(y \in Y\) and \(x \in X \setminus Y\).
\end{definition}

Our next goal is to remember the concept of \emph{metric reflection}.

A \emph{binary relation} on a set (or, more generally, on a class) \(X\) is a subset (a subclass) of the Cartesian square
\[
X^2 = X \times X = \{\<x, y>\colon x, y \in X\},
\]
where
\[
\<x, y> = \{\{x\}, \{x, y\}\}
\]
is an ordered (by Kazimierz Kuratowski) pair. A binary relation \(R \subseteq X^{2}\) is an \emph{equivalence relation} on \(X\) if the following conditions hold for all \(x\), \(y\), \(z \in X\):
\begin{enumerate}
\item \(\<x, x> \in R\), the \emph{reflexivity} property;
\item \((\<x, y> \in R) \Leftrightarrow (\<y, x> \in R)\), the \emph{symmetry} property;
\item \(((\<x, y> \in R) \text{ and } (\<y, z> \in R)) \Rightarrow (\<x, z> \in R)\), the \emph{transitivity} property.
\end{enumerate}

If \(R\) is an equivalence relation on a set \(X\), then an \emph{equivalence class} is a subset \([a]_R\) of \(X\) having the form
\begin{equation}\label{e1.1}
[a]_R = \{x \in X \colon \<x, a> \in R\}, \quad a \in X.
\end{equation}
The \emph{quotient set} of \(X\) with respect to \(R\) is the set of all equivalence classes \([a]_R\), \(a \in X\).

There exists the well-known, one-to-one correspondence between the equivalence relations on the sets and partitions of the sets.

Let \(X\) be a nonempty set and \(P = \{X_j \colon j \in J\}\) be a set of nonempty subsets of \(X\). Then \(P\) is a \emph{partition} of \(X\) with the \emph{blocks} \(X_j\) if
\[
\bigcup_{j \in J} X_j = X,
\]
and \(X_{j_1} \cap X_{j_2} = \varnothing\) holds for all distinct \(j_1\), \(j_2 \in J\).

\begin{proposition}\label{p2.2}
Let \(X\) be a nonempty set. If \(P = \{X_j \colon j \in J\}\) is a partition of \(X\) and \(R_P\) is a binary relation on \(X\) defined as
\begin{itemize}
\item[] \(\<x, y> \in R_P\) if and only if \(\exists j \in J\) such that  \(x \in X_j\) and \(y \in X_j\),
\end{itemize}
then \(R_P\) is an equivalence relation on \(X\) with the equivalence classes \(X_j\). Conversely, if \(R\) is an equivalence relation on \(X\), then the quotient set of \(X\) with respect to \(R\) is a partition of \(X\) with the blocks \([a]_R\).
\end{proposition}

(For the proof of Proposition~\ref{p2.2}, see, for example, \cite[Chapter~II, \S{}~5]{KurMost}.)

For every pseudometric space \((X, d)\), we define a binary relation \(\coherent{0}\) on \(X\) as
\begin{equation}\label{ch2:p1:e1}
(x \coherent{0} y) \Leftrightarrow (d(x, y) = 0), \quad x, y \in X
\end{equation}
and, similarly~\eqref{e1.1}, we write
\begin{equation}\label{e2.2}
[a]_0 := \{x \in X \colon d(a, x) = 0\},
\end{equation}
for every \(a \in X\).

\begin{proposition}\label{ch2:p1}
Let \(X\) be a nonempty set and let \(d \colon X^{2} \to \RR\) be a pseudometric on \(X\). Then \({\coherent{0}}\) is an equivalence relation on \(X\) and the function \(\delta_d\),
\begin{equation}\label{e1.1.5}
\delta_d(\alpha, \beta) := d(x, y), \quad x \in \alpha \in X/{\coherent{0}}, \quad y \in \beta \in X/{\coherent{0}},
\end{equation}
is a correctly defined metric on \(X/{\coherent{0}}\), where \(X/{\coherent{0}}\) is the quotient set of \(X\) with respect to \(\coherent{0}\).
\end{proposition}

The proof of Proposition~\ref{ch2:p1} can be found in \cite[Ch.~4, Th.~15]{Kelley1975}.
\medskip

In what follows we will say that the metric space \((X / \coherent{0}, \delta_d)\) is the \emph{metric reflection} of the pseudometric space \((X, d)\).

Let us introduce now the basic for us morphisms  of the pseudometric spaces. The following can be considered as a generalization of the concept of isometry of metric spaces.

\begin{definition}[{\cite{BD2021a}}]\label{d2.5}
Let \((X, d)\) and \((Y, \rho)\) be pseudometric spaces. A mapping \(\Phi \colon X \to Y\) is a \emph{pseudoisometry} of \((X, d)\) and \((Y, \rho)\) if:
\begin{enumerate}
\item \label{d2.11:s1} \(\rho(\Phi(x), \Phi(y)) = d(x, y)\) holds for all \(x\), \(y \in X\).
\item \label{d2.11:s2} For every \(u \in Y\) there is \(v \in X\) such that \(\rho(\Phi(v), u) = 0\).
\end{enumerate}
We say that two pseudometric spaces are \emph{pseudoisometric} if there is a pseudoisometry of these spaces.
\end{definition}

\begin{remark}\label{r1.12}
The concept of isometry of metric spaces can be extended to pseudometric spaces in various non-equivalent ways. John Kelley \cite{Kelley1975} define the isometries of pseudometric spaces \((X, d)\) and \((Y, \rho)\) as the distance-preserving surjections \(X \to Y\). It is clear that every isometry in Kelley's sense is a pseudoisometry. Another generalization of isometries is the combinatorial similarities of pseudometric spaces (see \cite{DLAMH2020, DovBBMSSS2020}).
\end{remark}

\begin{example}\label{ex2.5}
Let \((Y, \rho)\) be a pseudometric space and let \((Y / \coherent{0}, \delta_{\rho})\) be the metric reflection of \((Y, \rho)\). Then the mapping \(\pi_Y \colon Y \to Y / \coherent{0}\),
\begin{equation}\label{ex2.5:e1}
\pi_Y(y) = \{x \in Y \colon \rho(x, y) = 0\}, \quad y \in Y,
\end{equation}
is a pseudoisometry of the pseudometric space \((Y, \rho)\) and its metric reflection \((Y / \coherent{0}, \delta_{\rho})\). The Axiom of Choice implies the existence of a pseudoisometry \(\Psi \colon Y / \coherent{0} \to Y\) such that \(\pi_Y(\Psi(\alpha)) = \alpha\) for every \(\alpha \in Y / \coherent{0}\).
\end{example}

The following theorem is well-known (see, for example, Theorem~15 on page~123 of \cite{Kelley1975}).

\begin{theorem}\label{t1.14}
Let \((Y, \rho)\) be a nonempty pseudometric space and let \((Y / \coherent{0}, \delta_{\rho})\) be the metric reflection of \((Y, \rho)\). Then the following statements are equivalent:
\begin{enumerate}
\item \label{t1.14:c1} \((Y, \rho)\) is complete.
\item \label{t1.14:c2} \((Y/\coherent{0}, \delta_{\rho})\) is complete.
\end{enumerate}
\end{theorem}

\begin{remark}\label{r1.13}
It was proved in \cite{HK2015CMUC} that the Countable Axiom of Choice (\textbf{CAC}) is equivalent to the statement: ``A pseudometric space is complete iff its metric reflection is complete''.
\end{remark}

The results of the paper are presented as follows.

It will be shown in Theorem~\ref{t2.7} of Section~\ref{sec2} that there is no non-empty pseudometric space that would be both complete and closed in each of its pseudometric superspace. Moreover, Proposition~\ref{p2.6} shows that the equivalence \(\text{``closedness} = \text{completeness''}\) in the set of all subsets of a given pseudometric space is valid for the complete metric spaces.

Theorem~\ref{t2.9} gives us a characterization of the complete subsets of complete pseudometric spaces by describing the boundaries of these subsets, that can be considered as a pseudometric modification of Theorem~\ref{t1.10}. A dual of Theorem~\ref{t2.9} is given in Theorem~\ref{t2.10}.

Theorem~\ref{t2.13} describes the complete pseudometric spaces \((Y, \rho)\) as closed subspaces of \(\CEC(Y, \rho)\)-spaces. An extremal property of the classes \(\CEC(Y, \rho)\) is proved in Proposition~\ref{p2.15}.

In Theorem~\ref{t3.6}, we expand Theorem~\ref{t1.14} to arbitrary pseudoisometric spaces. The final result of the paper, Proposition~\ref{p3.7}, characterized the relation ``to be pseudoisometric'' by minimal property.

The results of the paper were motivated by Kelly's classical monograph \cite{Kelley1975} and by study of infinitesimal and asymptotic properties of metric spaces \cite{DD2010, DAK2010, DAK2011, DM2011RRMPA, BD2013BBMSSS, DAK2013, BD2014AASFM, ADK2016A, BD2018, BD2019JMSNY, BD2019JMSNYa, AK2021PdlM}.

\section{On completeness and closedness in pseudometric spaces}
\label{sec2}

Let us start from the following reformulation of Theorem~\ref{t1.14}.

\begin{lemma}\label{l2.1}
Let \((X, d)\) be a pseudometric space and let \(A\) be a nonempty subset of \(X\). Then \(A\) is complete in \((X, d)\) iff the set \(\cup_{a \in A} [a]_0\) is complete in \((X, d)\), where \([a]_0\) is defined by~\eqref{e2.2}.
\end{lemma}

\begin{proof}
The metric reflections of the spaces \(A\) and \(\cup_{a \in A} [a]_0\) are isometric metric spaces. Hence, they are either simultaneously complete or simultaneously incomplete. Now the conclusion of the lemma follows from Theorem~\ref{t1.14}.
\end{proof}

\begin{lemma}\label{l2.2}
Let \(A\) be a nonempty closed subset of a pseudometric space \((X, d)\). Then the equality
\begin{equation}\label{l2.2:e1}
A = \bigcup_{a \in A} [a]_0
\end{equation}
holds.
\end{lemma}

\begin{proof}
It is clear that \(A \subseteq \cup_{a \in A} [a]_0\). Thus, to prove \eqref{l2.2:e1}, it suffices to show that
\begin{equation}\label{l2.2:e2}
[a]_0 \subseteq A
\end{equation}
holds for every \(a \in A\). Let us do it.

Since \(A\) is closed in \((X, d)\), the set \(X \setminus A\) is open in \((X, d)\) and, consequently, there is \(\mbf{B}^1 \subseteq \mbf{B}_X\) such that
\begin{equation}\label{l2.2:e3}
X \setminus A = \bigcup_{B \in \mbf{B}^1} B.
\end{equation}
Using Proposition~\ref{p2.2}, we see that \eqref{l2.2:e2} holds for every \(a \in A\) if and only if
\begin{equation}\label{l2.2:e4}
[b]_0 \subseteq X \setminus A
\end{equation}
holds for every \(b \in X \setminus A\). From \eqref{l2.2:e3} it follows that \eqref{l2.2:e4} holds for every \(b \in X \setminus A\) if
\begin{equation}\label{l2.2:e5}
[b]_0 \subseteq B
\end{equation}
holds for every \(b \in B\) and every \(B \in \mbf{B}_X\). Now \eqref{l2.2:e5} can be directly proved by using the definition of open balls and the triangle inequality.
\end{proof}

For the case when \(A\) is a complete subset of \((X, d)\), equality~\eqref{l2.2:e1} is also sufficient for the closedness of \(A\).

\begin{proposition}\label{p2.7}
Let \((X, d)\) be a pseudometric space and let \(A\) be a nonempty complete subset of \((X, d)\). Then \(A\) is closed in \((X, d)\) if and only if
\begin{equation}\label{p2.7:e1}
[a]_0 \subseteq A
\end{equation}
holds for every \(a \in A\).
\end{proposition}

\begin{proof}
If \(A\) is closed, then \eqref{p2.7:e1} holds for every \(a \in A\) by Lemma~\ref{l2.2}. Let \eqref{p2.7:e1} hold for every \(a \in A\). If \(A\) is not a closed set, then \(X \setminus A\) is not an open set and, consequently, there is \(c \in X \setminus A\) such that
\[
B_r(c) \cap A \neq \varnothing
\]
for every \(r > 0\), where \(B_r(c)\) is the open ball with the center \(c\) and the radius \(r\). Consequently, there is a convergent sequence \(\seq{a_n}{n} \subseteq A\),
\begin{equation}\label{p2.7:e5}
\lim_{n \to \infty} d(a_n, c) = 0.
\end{equation}
Since \(A\) is complete and every convergent sequence is a Cauchy sequence, there is \(a \in A\) such that
\begin{equation}\label{p2.7:e6}
\lim_{n \to \infty} d(a_n, a) = 0.
\end{equation}
Using the triangle inequality and \eqref{p2.7:e5}--\eqref{p2.7:e6}, we obtain
\[
d(c, a) \leqslant \limsup_{n \to \infty} d(c, a_n) + \limsup_{n \to \infty} d(a_n, a) = 0.
\]
Therefore, \(c \in [a_0]\) holds, that implies \(c \in A\) by \eqref{p2.7:e1}. Thus, we obtain the contradiction, \(c \in A \cap (X \setminus A) = \varnothing\).
\end{proof}

\begin{corollary}\label{c2.8}
Let \((X, d)\) be a pseudometric space and let \(A\) be a nonempty complete subset of \(X\). Write \(\overline{A}\) for the closure of \(A\). Then \(\overline{A}\) is also a complete subset of \(X\) and the equality
\begin{equation}\label{c2.8:e0}
\overline{A} = \bigcup_{a \in A} [a]_0
\end{equation}
holds.
\end{corollary}

\begin{proof}
It follows from Lemma~\ref{l2.1} and Proposition~\ref{p2.7}.
\end{proof}

\begin{proposition}\label{p2.6}
Let \((X, d)\) be a pseudometric space. Then the following conditions are equivalent:
\begin{enumerate}
\item\label{p2.6:c1} \((X, d)\) is a metric space.
\item\label{p2.6:c2} The set of all closed subsets of \((X, d)\) coincides with the set of all complete subsets of \((X, d)\).
\end{enumerate}
\end{proposition}

\begin{proof}
\(\ref{p2.6:c1} \Rightarrow \ref{p2.6:c2}\). The validity of this implication follows from Theorem~\ref{t1.11}.

\(\ref{p2.6:c2} \Rightarrow \ref{p2.6:c1}\). Let \ref{p2.6:c2} hold. If the pseudometric \(d\) is not a metric, then there are different points \(a_0\), \(b_0 \in X\) such that
\begin{equation}\label{p2.6:e1}
d(a_0, b_0) = 0.
\end{equation}
The singleton \(A_0 = \{a_0\}\) is a complete subset of \((X, d)\). Hence, \eqref{c2.8:e0} implies the equality \(\ol{A_0} = [a_0]_0\). Equality \eqref{p2.6:e1} implies that \(b_0 \in \ol{A_0} \setminus A\). Thus, \(A_0\) is not closed subset in \((X, d)\). The validity of \(\ref{p2.6:c2} \Rightarrow \ref{p2.6:c1}\) is proved.
\end{proof}

\begin{theorem}\label{t2.7}
Let \((X, d)\) be a complete pseudometric space. Then the following conditions are equivalent:
\begin{enumerate}
\item \label{t2.7:c1} \(X\) is empty.
\item \label{t2.7:c2} \(X\) is a closed subset of every pseudometric superspace of \((X, d)\).
\end{enumerate}
\end{theorem}

\begin{proof}
\(\ref{t2.7:c1} \Rightarrow \ref{t2.7:c2}\). This implication is valid because the empty set is closed in all topological spaces.

\(\ref{t2.7:c2} \Rightarrow \ref{t2.7:c1}\). Let \(\ref{t2.7:c2}\) hold. Suppose that \(X\) is nonvoid and consider an arbitrary \(x_0 \in X\). Let \(y_0\) be a point such that \(y_0 \notin X\). Let us define a pseudometric superspace \((Y, \rho)\) of the space \((X, d)\) as follows. Write \(Y:= X \cup \{y_0\}\) and let \(\rho \colon Y^2 \to \RR\) be a symmetric mapping such that \(\rho|_{X^2} = d\) and
\begin{equation}\label{t2.7:e1}
\rho(y_0, x) = \begin{cases}
0 \text{ if } x = y_0,\\
d(x_0, x) \text{ if } x \in X.
\end{cases}
\end{equation}
Then \(\rho\) is a pseudometric on \(Y\). The equality \(\rho|_{X^2} = d\) and Definition~\ref{d2.6} imply that \(X\) is complete in \((Y, \rho)\). Hence, by Proposition~\ref{p2.7}, the inclusion
\begin{equation}\label{t2.7:e2}
[x_0]_0 \subseteq X
\end{equation}
holds with
\begin{equation}\label{t2.7:e3}
[x_0]_0 = \{y \in Y \colon \rho(y, x_0) = 0\}.
\end{equation}
Now using \eqref{t2.7:e1} and \eqref{t2.7:e3} we obtain \(y_0 \in [x_0]_0\), that implies \(y_0 \in X\). The last membership contradicts the condition \(y_0 \notin X\). The validity \(\ref{t2.7:c2} \Rightarrow \ref{t2.7:c1}\) follows.
\end{proof}

It is interesting to compare Theorem~\ref{t2.7} with the following result from~\cite{Mos2005IJMMS}.

\begin{theorem}\label{t2.8}
Let \((X, d)\) be a quasi-pseudometric space such that \(d\) is bounded and separates points. Then the set \(X\) is closed in each quasi-pseudometric superspace of \((X, d)\) if and only if \(X = \varnothing\).
\end{theorem}

The next our result is a natural extension of Theorem~\ref{t1.11} to the case of pseudometric spaces.

\begin{theorem}\label{t2.9}
Let \((X, d)\) be a complete pseudometric space. Then the following statements are equivalent for every \(A \subseteq X\).
\begin{enumerate}
\item\label{t2.9:s1} \(A\) is a complete subset of \((X, d)\).
\item\label{t2.9:s2} The intersection \([x]_0 \cap A\) is nonempty for every point \(x\) of the boundary \(\Fr(A)\).
\end{enumerate}
\end{theorem}

\begin{proof}
Since the boundary of the empty set \(\varnothing\) is empty and \(\varnothing\) is complete in any pseudometric space, the equivalence \(\ref{t2.9:s1} \Leftrightarrow \ref{t2.9:s2}\) is valid for \(A = \varnothing\). Let us consider the case when \(A\) is nonempty.

\(\ref{t2.9:s1} \Rightarrow \ref{t2.9:s2}\). Let \(A \subseteq X\) be complete in \((X, d)\), and let \(x\) be a point of \(\Fr(A)\). The formula
\begin{equation}\label{t2.9:e1}
\Fr(A) = \ol{A} \cap \ol{X \setminus A}
\end{equation}
implies that
\begin{equation}\label{t2.9:e2}
x \in \ol{A} \quad \text{and} \quad x \in \ol{X \setminus A}.
\end{equation}
The singleton \(\{x\}\) is a complete subset of \((X, d)\). Hence, by Proposition~\ref{p2.7}, the set \([x]_0\) is closed in \((X, d)\). Consequently, the closure of the set \(\{x\}\) is a subset of \(\ol{A}\) (because \(\ol{A} = \ol{\ol{A}}\) holds). Thus, we have the inclusion
\begin{equation}\label{t2.9:e3}
[x]_0 \subseteq \ol{A}.
\end{equation}
Since \(\ol{A} = \cup_{a \in A} [a]_0\) holds by Corollary~\ref{c2.8}, inclusion \eqref{t2.9:e3} implies
\[
[x]_0 \subseteq \cup_{a \in A} [a]_0.
\]
Thus, there is \(a \in A\) such that
\begin{equation}\label{t2.9:e4}
[x]_0 \cap [a]_0 \neq \varnothing.
\end{equation}
By Proposition~\ref{ch2:p1}, for any two different \(p\), \(q \in X\), we have either \([p]_0 = [q]_0\) or \([p]_0 \cap [q]_0 = \varnothing\). Hence, \eqref{t2.9:e4} implies that \([x]_0 = [a]_0\) holds. Thus, the point \(a \in A\) belongs to intersection \([x]_0 \cap A\).

\(\ref{t2.9:s2} \Rightarrow \ref{t2.9:s1}\). Let \ref{t2.9:s2} hold. We must prove that \(A\) is a complete subset of \((X, d)\). Let us consider an arbitrary Cauchy sequence \((a_n)_{n \in \NN} \subseteq A\). Since \(A\) is a subset of \(X\) and \((X, d)\) is complete, there is a point \(b \in X\) such that
\begin{equation}\label{t2.9:e5}
\lim_{n \to \infty} d(a_n, b) = 0.
\end{equation}
It suffices to show that there is a point \(a\) of the set \(A\) such that \(a \in [b]_0\). The last statement is evidently holds if \(b \in A\).

Suppose that \(b\) belongs to \(X \setminus A\). Then we have \(b \in \ol{X \setminus A}\) and, in addition, from \eqref{t2.9:e5} it follows that \(b \in \ol{A}\). Hence, \(b\) belongs to \(\Fr(A)\) by \eqref{t2.9:e1}. Now using \ref{t2.9:s2} we can find a point \(a\) of \(A\) such that \(a \in [b]_0\).
\end{proof}

\begin{remark}\label{r2.11}
It is easy to see that Theorem~\ref{t2.9} implies Theorem~\ref{t1.11}. Indeed, if \((X, d)\) is a metric space, then statement \ref{t2.9:s2} of Theorem~\ref{t2.9} holds iff we have \(\Fr(A) \subseteq A\) but the last inclusion means that \(A\) is closed in \((X, d)\).
\end{remark}

\begin{lemma}\label{l2.6}
Every closed subset of each complete pseudometric space is complete.
\end{lemma}

For the proof see Theorem~22 on page~192 of \cite{Kelley1975}.

The following result gives us a dual form of Theorem~\ref{t2.9}.

\begin{theorem}\label{t2.10}
Let \((X, d)\) be a complete pseudometric space. Then the following statements are equivalent for every nonempty \(A \subseteq X\):
\begin{enumerate}
\item \label{t2.10:s1} \(A\) is a closed subset of \((X, d)\).
\item \label{t2.10:s2} The set \(A\) is a complete subset of \((X, d)\) and the equality
\begin{equation}\label{t2.10:e1}
A = \bigcup_{a \in A} [a]_0
\end{equation}
holds.
\end{enumerate}
\end{theorem}

\begin{proof}
\(\ref{t2.10:s1} \Rightarrow \ref{t2.10:s2}\). Let \(\ref{t2.10:s1}\) hold. Then equality \eqref{t2.10:e1} follows from Lemma~\ref{l2.2}. Moreover, \(A\) is complete by Lemma~\ref{l2.6}.

\(\ref{t2.10:s2} \Rightarrow \ref{t2.10:s1}\). Let \(\ref{t2.10:s2}\) hold. Then, by Corollary~\ref{c2.8}, the equality
\[
\ol{A} = \bigcup_{a \in A} [a]_0
\]
holds. The last equality and \eqref{t2.10:e1} imply \(A = \ol{A}\). Thus, \(A\) is closed.
\end{proof}

Let us now turn to a pseudometric analogue of Theorem~\ref{t1.10}.

\begin{theorem}\label{t2.13}
Let \((Y, \rho)\) be a nonempty pseudometric space. Then the following statements are equivalent:
\begin{enumerate}
\item \label{t2.13:s1} \((Y, \rho)\) is complete.
\item \label{t2.13:s2} \(Y\) is a closed subset of every \((X, d) \in \CEC(Y, \rho)\).
\end{enumerate}
\end{theorem}

\begin{proof}
\(\ref{t2.13:s1} \Rightarrow \ref{t2.13:s2}\). Let \((Y, \rho)\) be a complete pseudometric space and let \((X, d) \in \CEC(Y, \rho)\). We must show that \(Y\) is closed in \((X, d)\).

By Proposition~\ref{p2.7}, the set \(Y\) is closed in \((X, d)\) iff the inclusion
\begin{equation}\label{t2.13:e1}
\{x \in X \colon d(x, y) = 0\} \subseteq Y
\end{equation}
holds for every \(y \in Y\). Since \((X, d)\) belongs to \(\CEC(Y, \rho)\), we have the inequality \(d(x, y) > 0\) for all \(y \in Y\) and \(x \in X \setminus Y\). Hence, \eqref{t2.13:e1} is equivalent to the trivially valid inclusion
\[
\{x \in Y \colon \rho(x, y) = 0\} \subseteq Y.
\]

\(\ref{t2.13:s2} \Rightarrow \ref{t2.13:s1}\). Let \ref{t2.13:s2} hold. Then, to prove the validity of \ref{t2.13:s1}, it suffices to find a complete \((X, d) \in \CEC(Y, \rho)\) and to use Lemma~\ref{l2.6}.

To construct a complete \((X, d) \in \CEC(Y, \rho)\), we consider the metric reflection \((Y/ \coherent{0}, \delta_{\rho})\) of the pseudometric space \((Y, \rho)\) and the (metric) completion \((Y^*, \rho^*)\) of \((Y/\coherent{0}, \delta_{\rho})\) such that \((Y^*, \rho^*)\) is a (metric) superspace of \((Y/\coherent{0}, \delta_{\rho})\). Without loss of generality, we can assume that if \(y^*\) is a point of \(Y^*\) such that \(y^* \notin Y/\coherent{0}\), then \(y^* \notin Y\) also holds. As in Example~\ref{ex2.5}, we consider the mapping \(\pi_Y \colon Y \to Y/\coherent{0}\) by
\[
\pi_Y(y) = \{x \in Y \colon \rho(x, y) = 0\}, \quad y \in Y.
\]

Let us define a set \(X\) and a pseudometric \(d\) on \(X\) by the rules:
\begin{gather}
\label{t2.13:e2}
(x \in X) \Leftrightarrow \left(x \in Y \cup Y^* \text{ and } x \notin Y/\coherent{0}\right),\\
\label{t2.13:e3}
d(x_1, x_2) = \begin{cases}
\rho^*(\pi_Y(x_1), \pi_Y(x_2)) & \text{if } x_1, x_2 \in Y,\\
\rho^*(\pi_Y(x_1), x_2) & \text{if } x_1 \in Y \text{ and } x_2 \in X \setminus Y,\\
\rho^*(x_1, \pi_Y(x_2)) & \text{if } x_1\in X \setminus Y \text{ and } x_2 \in Y,\\
\rho^*(x_1, x_2) & \text{if } x_1, x_2 \in X \setminus Y.
\end{cases}
\end{gather}
Then \((X, d)\) is complete and belongs to \(\CEC(Y, \rho)\).
\end{proof}

\begin{remark}\label{r2.14}
If \((Y, \rho)\) is a metric space, then \(\CEC(Y, \rho)\) includes the class of all metric superspaces of \((Y, \rho)\). Consequently, Theorem~\ref{t2.13} can be considered as a generalization of Theorem~\ref{t1.10}.
\end{remark}

\begin{remark}
Each pseudometric space \((Y, \rho)\) can be mapped by one-to-one pseudoisometry onto a dense subset of a complete pseudometric space \((X, d)\) (see Theorem~27 on page~196 of \cite{Kelley1975}). If we define \((X, d)\) by rules \eqref{t2.13:e2}--\eqref{t2.13:e3}, then \((Y, \rho)\) is a dense subset of \((X, d)\) and \((X, d)\) can be considered as a pseudometric completion of \((Y, \rho)\).
\end{remark}

\begin{proposition}\label{p2.15}
Let \((Y, \rho)\) be a complete pseudometric space and let \(\CEC^*(Y, \rho)\) be a class of pseudometric superspaces of \((Y, \rho)\) such that \(Y\) is closed in every \((X, d) \in \CEC^*(Y, \rho)\). Then the inclusion
\[
\CEC^*(Y, \rho) \subseteq \CEC(Y, \rho)
\]
holds.
\end{proposition}

\begin{proof}
Let \((X, d)\) be a pseudometric space such that \((X, d) \in \CEC^*(Y, \rho)\) and \((X, d) \notin \CEC(Y, \rho)\). Then there are \(x \in X \setminus Y\) and \(y \in Y\) satisfying the equality \(d(x, y) = 0\). Hence, by Proposition~\ref{p2.7}, the set \(Y\) is not a closed subset of \((X, d)\), contrary to the definition of \(\CEC^*(Y, \rho)\).
\end{proof}

\section{Metric reflections and pseudoisometries}

The following properties of pseudoisometries are easy to prove.

\begin{proposition}\label{p2.5}
Let \((X, d)\) and \((Y, \rho)\) be pseudometric spaces. Then the following statements hold for every pseudoisometry \(\Phi \colon X \to Y\):
\begin{enumerate}
\item \label{p2.5:s1} If \((X, d)\) and \((Y, \rho)\) are metric spaces, then \(\Phi\) is an isometry.
\item \label{p2.5:s2} If \((Y, \rho)\) is a metric space, then \(\Phi\) is a surjection.
\item \label{p2.5:s3} If \((X, d)\) is a metric space, then \(\Phi\) is an injection.
\item \label{p2.5:s4} \(\Phi\) is a homeomorphism of the topological spaces \(X\) and \(Y\) (endowed with topologies generated by \(d\) and \(\rho\), respectively) if and only if \(\Phi\) is bijective.
\end{enumerate}
\end{proposition}

\begin{lemma}\label{l2.8}
Let \((X, d)\), \((Y, \rho)\) and \((Z, \delta)\) be pseudometric spaces and let \(\Phi \colon X \to Y\) and \(\Psi \colon Y \to Z\) be pseudoisometries. Then the mapping
\[
X \xrightarrow{\Phi} Y \xrightarrow{\Psi} Z
\]
is also a pseudoisometry.
\end{lemma}

\begin{proof}
It follows directly from Definition~\ref{d2.5}.
\end{proof}

\begin{theorem}\label{t3.3}
Let \((X, d)\) and \((Y, \rho)\) be pseudometric spaces. Then \((X, d)\) and \((Y, \rho)\) are pseudoisometric if and only if the metric reflections \((X / \coherent{0}, \delta_d)\) and \((Y / \coherent{0}, \delta_\rho)\) are isometric metric spaces.
\end{theorem}

\begin{proof}
Suppose that \((X, d)\) and \((Y, \rho)\) are pseudoisometric. We will prove that \((X / \coherent{0}, \delta_d)\) and \((Y / \coherent{0}, \delta_{\rho})\) are isometric.

Write \(\Phi \colon X \to Y\) for a pseudoisometry of \((X, d)\) and \((Y, \rho)\). Let \(\pi_X \colon X \to X / \coherent{0}\) and \(\pi_Y \colon Y \to Y / \coherent{0}\) be the natural projection of \(X\) and \(Y\) on the quotient sets \(X / \coherent{0}\) and \(Y / \coherent{0}\), respectively, i.e.,
\begin{align}\label{t3.3:e1}
\pi_X (x) &= \{z \in X \colon d(x, z) = 0\},\\
\label{t3.3:e2}
\pi_Y (y) &= \{t \in Y \colon \rho(t, y) = 0\}
\end{align}
hold for all \(x \in X\) and \(y \in Y\).

We claim that there is a mapping \(F \colon X / \coherent{0} \to Y / \coherent{0}\) such that the diagram
\begin{equation}\label{t3.3:e3}
\ctdiagram{
\ctv 0,30: {X}
\ctv 120,30: {Y}
\ctv 0,-30: {X / \coherent{0}}
\ctv 120,-30: {Y / \coherent{0}}
\ctet 0,30,120,30:{\Phi}
\ctet 0,-30,120,-30:{F}
\ctel 0,30, 0,-30:{\pi_X}
\cter 120,30, 120,-30:{\pi_Y}
}
\end{equation}
is commutative, i.e., \(F \circ \pi_X = \pi_Y \circ \Phi\) holds. A desired \(F\) can be found if and only if the implication
\begin{equation}\label{t3.3:e4}
\bigl(\pi_X(x_1) = \pi_X(x_2)\bigr) \Rightarrow \bigl(\pi_Y(\Phi(x_1)) = \pi_Y(\Phi(x_2))\bigr)
\end{equation}
is valid for all \(x_1\), \(x_2 \in X\). Let us prove the validity of \eqref{t3.3:e4}. If we have \(\pi_X(x_1) = \pi_X(x_2)\), then \eqref{t3.3:e1} implies \(d(x_1, x_2) = 0\). The last equality and statement \ref{d2.11:s1} from Definition~\ref{d2.5} give us the equality
\begin{equation}\label{t3.3:e5}
\rho(\Phi(x_1), \Phi(x_2)) = 0.
\end{equation}
Now from \eqref{t3.3:e5} and \eqref{t3.3:e2} it follows that \(\pi_Y(\Phi(x_1)) = \pi_Y(\Phi(x_2))\). Thus, there is \(F \colon X / \coherent{0} \to Y / \coherent{0}\) such that diagram \eqref{t3.3:e3} is commutative.

Let \(F \colon X / \coherent{0} \to Y / \coherent{0}\) satisfy
\begin{equation}\label{t3.3:e10}
F \circ \pi_X = \pi_Y \circ \Phi
\end{equation}
and let \(\alpha\), \(\beta\) be arbitrary points of \(X / \coherent{0}\). Then, by \eqref{e1.1.5}, we have
\begin{equation}\label{t3.3:e6}
\delta_d (\alpha, \beta) = d(x_1, x_2)
\end{equation}
for all \(x_1 \in \alpha\) and \(x_2 \in \beta\). Using \eqref{t3.3:e1} we can rewrite \eqref{t3.3:e6} as
\begin{equation}\label{t3.3:e7}
\delta_d (\pi_X(x_1), \pi_X(x_2)) = d(x_1, x_2).
\end{equation}
The last equality and statement \ref{d2.11:s1} from Definition~\ref{d2.5} imply
\begin{equation}\label{t3.3:e8}
d(x_1, x_2) = \rho (\Phi(x_1), \Phi(x_2)).
\end{equation}
Now using Proposition~\ref{ch2:p1} with \(X = Y\) and \(d = \rho\), we obtain
\begin{equation}\label{t3.3:e9}
\rho (\Phi(x_1), \Phi(x_2)) = \delta_{\rho}(\pi_Y(\Phi(x_1)), \pi_Y(\Phi(x_2))).
\end{equation}
Consequently, we have
\begin{equation}\label{t3.3:e11}
\delta_d (\pi_X(x_1), \pi_X(x_2)) = \delta_{\rho}(\pi_Y(\Phi(x_1)), \pi_Y(\Phi(x_2))).
\end{equation}
Using \eqref{t3.3:e10} we can rewrite \eqref{t3.3:e11} as
\[
\delta_d (\pi_X(x_1), \pi_X(x_2)) = \delta_{\rho}(F(\pi_X(x_1)), F(\pi_X(x_2))).
\]
Thus, we have
\[
\delta_d (\alpha, \beta) = \delta_{\rho} (F(\alpha), F(\beta))
\]
for all \(\alpha\), \(\beta \in X / \coherent{0}\). Consequently, \(F\) is an isometric embedding of the metric space \((X/\coherent{0}, \delta_d)\) in the metric space \((Y/\coherent{0}, \delta_{\rho})\). Moreover, by Example~\ref{ex2.5}, the mapping \(\pi_Y \colon Y \to Y / \coherent{0}\) is a pseudoisometry and, consequently,
\[
X \xrightarrow{\Phi} Y \xrightarrow{\pi_Y} Y / \coherent{0}
\]
is a pseudoisometry by Lemma~\ref{l2.8}. Since \((Y/\coherent{0}, \delta_{\rho})\) is a metric space, the pseudoisometry \(X \xrightarrow{\Phi} Y \xrightarrow{\pi_Y} Y / \coherent{0}\) is surjective by Statement~\ref{p2.5:s2} of Proposition~\ref{p2.5}. Using equality~\eqref{t3.3:e10}, we obtain that \(F \circ \pi_X\) is also surjective, that implies the surjectivity of \(F\). Every surjective isometric embedding is an isometry. Thus, \(F\) is an isometry of \((X/\coherent{0}, \delta_d)\) and \((Y/\coherent{0}, \delta_{\rho})\).

Conversely, let \((X/\coherent{0}, \delta_d)\) and \((Y/\coherent{0}, \delta_{\rho})\) be isometric metric spaces. Write \(F \colon X/\coherent{0} \to Y/\coherent{0}\) for an isometry of these spaces. Let \(\pi_X \colon X \to X/\coherent{0}\) be the natural projection defined by \eqref{t3.3:e1} and let \(\Psi \colon Y/\coherent{0} \to Y\) be defined as in Example~\ref{ex2.5}. Then, using the same example and Lemma~\ref{l2.8}, we see that the mapping
\[
X \xrightarrow{\pi_X} X/\coherent{0} \xrightarrow{F} Y/\coherent{0} \xrightarrow{\Psi} Y
\]
is a pseudoisometry of \((X, d)\) and \((Y, \rho)\).
\end{proof}

Let us turn now to some corollaries of Theorem~\ref{t3.3}.

Write \(R_{\Pi}\) for the relation ``to be pseudoisometric''.

\begin{corollary}\label{c2.7}
\(R_{\Pi}\) is an equivalence relation on the class of all pseudometric spaces.
\end{corollary}

\begin{proof}
The relation ``to be isometric'' is an equivalence relation on the class of all metric spaces.
\end{proof}

Analyzing the proof of Theorem~\ref{t3.3}, we obtain the following.

\begin{corollary}\label{c2.13}
Let \(\Phi \colon X \to Y\) be a pseudoisometry of pseudometric spaces \((X, d)\) and \((Y, \rho)\). If \(F \colon X / {\coherent{0}} \to Y / {\coherent{0}}\) is a mapping such that diagram~\eqref{t3.3:e3} is commutative, then \(F\) is an isometry of the metric spaces \((X / {\coherent{0}}, \delta_d)\) and \((Y / {\coherent{0}}, \delta_{\rho})\).
\end{corollary}

Now we can simply expand Theorem~\ref{t1.14}.

\begin{theorem}\label{t3.6}
Let \((X, d)\) and \((Y, \rho)\) be pseudoisometric spaces. Then \((X, d)\) is complete if and only if \((Y, \rho)\) is complete.
\end{theorem}

\begin{proof}
Since \((X, d)\) and \((Y, \rho)\) are pseudoisometric, the metric spaces \((X/\coherent{0}, \delta_d)\) and \((Y/\coherent{0}, \delta_\rho)\) are isometric by Theorem~\ref{t3.3}. Hence, \((X/\coherent{0}, \delta_d)\) is complete iff \((Y/\coherent{0}, \delta_\rho)\) is complete. Now using Theorem~\ref{t1.14} we see that \((X, d)\) is complete iff \((Y, \rho)\) is complete.
\end{proof}

The following proposition characterizes the relation \(R_{\Pi}\) by a ``minimal'' property.

\begin{proposition}\label{p3.7}
Let \(R\) be an equivalence relation of the class of all pseudometric spaces. If we have the inclusion
\begin{equation}\label{p3.7:e1}
R \subseteq R_{\Pi},
\end{equation}
and any two isometric metric spaces are \(R\)-equivalent, and \(\<(Y, \rho), (Y/\coherent{0}, \delta_\rho)> \in R\) is valid for every pseudometric space \((Y, \rho)\), then the equality \(R = R_{\Pi}\) holds.
\end{proposition}

\begin{proof}
Suppose that inclusion \eqref{p3.7:e1} holds, and any two isometric metric spaces are \(R\)-equivalent, and each pseudometric space is \(R\)-equivalent to its metric reflection. We must show that
\begin{equation}\label{p3.7:e3}
\<(X, d), (Y, \rho)> \in R
\end{equation}
holds whenever
\begin{equation}\label{p3.7:e2}
\<(X, d), (Y, \rho)> \in R_{\Pi}.
\end{equation}
To prove the validity of \eqref{p3.7:e3}, we note that the memberships
\begin{equation}\label{p3.7:e4}
\<(X, d), (X/\coherent{0}, \delta_d)> \in R
\end{equation}
and
\begin{equation}\label{p3.7:e5}
\<(Y, \rho), (Y/\coherent{0}, \delta_\rho)> \in R
\end{equation}
are valid by supposition above. Moreover, by Corollary~\ref{c2.13}, membership~\eqref{p3.7:e2} implies that \((X/\coherent{0}, \delta_d)\) and \((Y/\coherent{0}, \delta_\rho)\) are isometric metric spaces. Thus, we have
\begin{equation}\label{p3.7:e6}
\<(X/\coherent{0}, \delta_d), (Y/\coherent{0}, \delta_\rho)> \in R
\end{equation}
by our supposition. Now \eqref{p3.7:e3} follows from \eqref{p3.7:e4}--\eqref{p3.7:e6} by transitivity of \(R\).
\end{proof}

\section*{Funding}

Oleksiy Dovgoshey was partially supported by Volkswagen Stiftung Project ``From
Modeling and Analysis to Approximation''.

\end{document}